\documentclass[11pt]{amsart}
\usepackage{amsmath,amssymb}

\begin{document}
\newcommand{\chp}{\mathds{1}}

\newtheorem{thm}{Theorem}[section]
\newtheorem{theorem}{Theorem}[section]
\newtheorem{lem}[thm]{Lemma}
\newtheorem{lemma}[thm]{Lemma}
\newtheorem{prop}[thm]{Proposition}
\newtheorem{proposition}[thm]{Proposition}
\newtheorem{cor}[thm]{Corollary}
\newtheorem{defn}[thm]{Definition}
\newtheorem*{remark}{Remark}
\newtheorem{conj}[thm]{Conjecture}

%\numberwithin{equation}{section}

\newcommand{\Z}{{\mathbb Z}} %cph changed from \mathbf
\newcommand{\Q}{{\mathbb Q}}
\newcommand{\R}{{\mathbb R}}
\newcommand{\C}{{\mathbb C}}
\newcommand{\N}{{\mathbb N}}
\newcommand{\FF}{{\mathbb F}}
\newcommand{\fq}{\mathbb{F}_q}
\newcommand{\rmk}[1]{\footnote{{\bf Comment:} #1}}

\newcommand{\bfA}{{\boldsymbol{A}}}
\newcommand{\bfY}{{\boldsymbol{Y}}}
\newcommand{\bfX}{{\boldsymbol{X}}}
\newcommand{\bfZ}{{\boldsymbol{Z}}}
\newcommand{\bfa}{{\boldsymbol{a}}}
\newcommand{\bfy}{{\boldsymbol{y}}}
\newcommand{\bfx}{{\boldsymbol{x}}}
\newcommand{\bfz}{{\boldsymbol{z}}}
\newcommand{\F}{\mathcal{F}}
\newcommand{\Gal}{\mathrm{Gal}}
\newcommand{\Fr}{\mathrm{Fr}}
\newcommand{\Hom}{\mathrm{Hom}}
\newcommand{\GL}{\mathrm{GL}}

\renewcommand{\mod}{\;\operatorname{mod}}
\newcommand{\ord}{\operatorname{ord}}
\newcommand{\TT}{\mathbb{T}}
\renewcommand{\i}{{\mathrm{i}}}
\renewcommand{\d}{{\mathrm{d}}}
\renewcommand{\^}{\widehat}
\newcommand{\HH}{\mathbb H}
\newcommand{\Vol}{\operatorname{vol}}
\newcommand{\area}{\operatorname{area}}
\newcommand{\tr}{\operatorname{tr}}
\newcommand{\norm}{\mathcal N} % norm =(\frac{ n+\sqrt{n^2-4}} 2)^2
\newcommand{\intinf}{\int_{-\infty}^\infty}
\newcommand{\ave}[1]{\left\langle#1\right\rangle} %  average
\newcommand{\Var}{\operatorname{Var}}
\newcommand{\Prob}{\operatorname{Prob}}
\newcommand{\sym}{\operatorname{Sym}}
\newcommand{\disc}{\operatorname{disc}}
\newcommand{\CA}{{\mathcal C}_A}
\newcommand{\cond}{\operatorname{cond}} % conductor
\newcommand{\lcm}{\operatorname{lcm}}
\newcommand{\Kl}{\operatorname{Kl}} %Kloosterman sum
\newcommand{\leg}[2]{\left( \frac{#1}{#2} \right)}  % Legendre symbol
\newcommand{\Li}{\operatorname{Li}}

\newcommand{\sumstar}{\sideset \and^{*} \to \sum}
\newcommand{\prodstar}{{\sideset \and^{*} \to \prod}}

\newcommand{\LL}{\mathcal L} %L-function of u
\newcommand{\sumf}{\sum^\flat}
\newcommand{\Hgev}{\mathcal H_{2g+2,q}}
\newcommand{\USp}{\operatorname{USp}}
\newcommand{\conv}{*}
\newcommand{\dist} {\operatorname{dist}}
\newcommand{\CF}{c_0} % Fejer constant
\newcommand{\kerp}{\mathcal K}

\newcommand{\Cov}{\operatorname{cov}}
\newcommand{\Sym}{\operatorname{Sym}}

\newcommand{\ES}{\mathcal S} % sums over AP
\newcommand{\EN}{\mathcal N} % sum over short intervals
\newcommand{\EM}{\mathcal M} % sum over pols of deg n
\newcommand{\Sc}{\operatorname{Sc}} %Secular coefficients
\newcommand{\Ht}{\operatorname{Ht}}

\newcommand{\E}{\operatorname{E}} % expectation
\newcommand{\sign}{\operatorname{sign}} %sign

\newcommand{\divid}{d} % the divisor function

\newcommand{\gl}{L}

% for the appendix
\newcommand{\I}{\mathcal A}
\newcommand{\II}{\mathcal B}
\newcommand{\III}{\mathcal C}
\newcommand{\sI}{\Sigma_\mathcal A}
\newcommand{\sII}{\Sigma_\mathcal B}
\newcommand{\sIII}{\Sigma_\mathcal C}
%%%%%%%%%%%%%
\title[Locally repeated values of arithmetic functions]{On locally repeated values of arithmetic functions over $\fq[T]$ }

\author{Ze\'ev Rudnick}

\address{Raymond and Beverly Sackler School of Mathematical Sciences,
Tel Aviv University, Tel Aviv 69978, Israel}
%\email{rudnick@post.tau.ac.il}
 
\date{\today}
 \dedicatory{\rm \large with an Appendix by Ron Peled}

\begin{abstract}
The frequency of occurrence of ``locally repeated" values of arithmetic functions  is a common theme in analytic number theory, for instance in the  Erd\H{o}s-Mirsky problem on coincidences of the divisor function at consecutive integers, the analogous problem for the Euler totient function, and the quantitative conjectures of Erd\H{o}s, Pomerance and Sark\H{o}zy and of  Graham, Holt and Pomerance on the frequency of occurrences. 
In this paper we introduce  the corresponding problems in the setting of polynomials over a finite field, and completely solve them in the large finite field limit.

\end{abstract}

\maketitle

%\tableofcontents

\section{Introduction}

The frequency of occurrence of ``locally repeated" values of arithmetic functions  is a common theme in analytic number theory. For instance, the famous Erd\H{o}s-Mirsky problem asked to show that there are infinitely many integers $n$ for which $d(n)=d(n+1)$, where $d(n)$ is the divisor function; and the number of such occurrences was the subject of a series of papers of Erd\H{o}s,  Pomerance and Sark\H{o}zy. One can replace the divisor function $d(n)$ by the number $\omega(n)$ of distinct prime divisors of $n$; the problems turn out to be closely related. Replacing $d(n)$ by the Euler totient function $\varphi(n)$ leads to somewhat different problems.  These questions  have generated a large body of literature (some described below), with several open conjectures.  In this paper we introduce  the corresponding problems in the setting of polynomials over a finite field, and completely solve them in the large finite field limit.

\subsection{The problem of Erd\H{o}s and Mirsky} 
We begin with an account of the state of the art for the problems over the integers.

The problem of Erd\H{o}s and Mirsky \cite{EM} is to show that there are infinitely many  integers $n$ such that $\divid(n)=\divid(n+1)$ where $\divid(n)$ is the number of divisors of $n$. 
This was proved by Heath-Brown \cite{HB1984} following an idea in Spiro's Ph.D. thesis \cite{Spiro}, who showed that $\divid(n)=\divid(n+5040)$ has infinitely many solutions, and Pinner  \cite{Pinner} showed that for any $k\geq 1$, there are infinitely many integers $n$ with $\divid(n) = \divid(n+k)$.   We now know much more, for instance  Graham, Goldston, Pintz and Yildirim \cite{GGPY} show that there are infinitely many $n$'s so that both $n$ and $n+1$ have prime factorizations of the form $p_1^2p_2p_3p_4$ with $p_j$ distinct primes, hence $d(n)=24=d(n+1)$. 

The same problem arises for other arithmetic functions, such as $\Omega(n)$, the number of all prime divisors of $n$   \cite{HB1984}, or $\omega(n)$, the number of distinct prime divisors of $n$ \cite{SP}.  
One can also ask  about multiple shifts, for instance are there infinitely many solutions of 
$$d(n)=d(n+1)=d(n+2)$$
of which nothing is currently known.

 The quantitative aspect of the problem is to find the asymptotic of 
$$
S_\alpha(x):=\#\Big\{n\leq x: \alpha(n)=\alpha(n+1)\Big\}
$$
where $\alpha=d,\omega,\Omega$, and likewise for any shift.  
   Erd\H{o}s, Pomerance and Sark\H{o}zy  \cite{EPS2} conjectured that the order of magnitude is given by 
\begin{equation}\label{conj EPS}
S_\alpha(x) \sim \frac 1{2\sqrt{\pi}} \frac{x}{\sqrt{\log \log x}}
\end{equation}
 They proved an  upper bound  $S_\alpha(x)\ll  x/\sqrt{\log \log x}$  of the correct order of magnitude \cite{EPS3} 
%(also for $\omega$ and $\Omega$)  
and there is a lower bound which is not far from the upper bound $S_\alpha(x)\gg x/(\log\log x)^3 $, due to Hildebrand \cite{Hildebrand}. 
%, improving on  \cite{HB1984}: 

One can more generally ask for the asymptotic frequency of coincidences of any number of shifts, that is given any distinct integers $a_1,\dots, a_r$, for 
$$
S_\alpha(a_1\dots,a_r;x):=\#\Big\{n\leq x: \alpha(n+a_1)=\dots = \alpha(n+a_r)\Big\}
$$
and using the same heuristic as in \cite{EPS2}, namely that the shifts are statistically independent, combined with the Erd\H{o}s-Kac theorem, one is led to conjecture that 
\begin{equation}\label{conj EPSmultiple}
S_\alpha(a_1\dots,a_r;x) \sim \frac 1{\sqrt{r}(\sqrt{2\pi})^{r-1}} \frac{x}{ (\log \log x)^{(r-1)/2}}
\end{equation}

\subsection{Locally repeated value of the Euler totient function}
Given a nonzero integer $k\geq 1$, it was conjectured  in \cite{EPS2} that there are infinitely many integers $n$ for which $\varphi(n)=\varphi(n+k)$. This is not known for any value of $k$. 
Let $P(k,x)$ be the number of such integers  $n\leq x$: 
$$P(k,x):=\#\{n\leq x: \varphi(n)=\varphi(n+k)\}
\;.
$$    
For instance, when $x=10^8$, we have \cite{GHP} 
$$
P(1,10^8) = 306, \quad P(2,10^8) = 125986, \quad P(3,10^8)=2, \quad P(4,10^8)=69131
$$
In \cite{EPS2, GHP} it is shown that    $P(k,x)=o(x)$ for any $k\geq 1$. 

There is a significant difference between $k$ being even or odd, due to the ability to find solutions to the problem when $k$ is even:  
Leo Moser \cite{Moser} observed that for integers of the form $n=2(2p-1)$ where $p$ is a prime such that $2p-1$ is also prime, then $\varphi(n)=2p-2=\varphi(4p)=\varphi(n+2)$, and Schinzel \cite{Schinzel} extended this observation to the family $n=k(2p-1)$ where $p$ is a  prime for which $2p-1$ is also prime, and both are coprime to $k$. 
Therefore assuming a suitable quantitative version of the twin prime conjectures gives at least $\gg x/(\log x)^2$ solutions when $k$ is even. However, when $k$  is odd, we have a smaller upper bound 
$$P(k,x)\ll x/\exp\{(\log x)^{1/3}\},\qquad k \; {\rm odd}
\;.
$$
It is also conjectured \cite{EPS2} that for any $k\geq 1$,  there is a lower bound of $P(k,x)\gg x^{1-\epsilon}$.

Graham, Holt and Pomerance \cite{GHP} systematized these observations and used them to conjecture that for $k$ even,  
\begin{equation}\label{conj k even}
\frac 1x P(k,x)\sim  \frac{A(k)}{(\log x)^2} \quad {\rm as}\quad x\to \infty\;,
 \end{equation}
where $A(k)=2C_2 \cdot  c(k) $, with $C_2=\prod_{p>2}(1-(p-1)^{-2})=0.6601\dots$ is the twin prime constant, and 
$$ c(k) = \sum_{\substack{j,j+k\\{\rm   same \; prime \; divisors}
}}\frac{\gcd(j,j+k)}{j(j+k)} \sideset{}{^*}\prod_p \frac{p-1}{p-2}
$$
where the sum is over all $j$'s so that $j$ and $j+k$ have the same prime factors, and the product is over primes $p>2$ dividing $jk(j+k)/\gcd(j,j+k)^3$. For instance, $c(2)=1/2$. 

One can more generally ask the same question for multiple shifts, and to replace the Euler totient function $\varphi$ by the sum-of divisors function $\sigma(n) = \sum_{d\mid n} d$.

\subsection{The problem of  Erd\H{o}s and Mirsky over $\fq[T]$}

Let $\fq$ be a finite field of $q$ elements, and $\fq[T]$ the ring of polynomials with coefficients in $\fq$. For $n\geq 0$ let $M_n\subset \fq[T]$ be the set of monic polynomials of degree $n$. Let $\alpha$ be an arithmetic function, that is a complex-valued function on the set of monic polynomials. 
%which is  {\em even}:   $\alpha(cf) = \alpha(f)$, for $c\in \fq^\times$ and $0\neq f\in \fq[T]$.
 For each finite field $\fq$, we are given $r$ distinct polynomials $a_1,\dots, a_r\in \fq[T]$ of degree $<n$.  
 We want to compute the probability that $\alpha(f+a_1) = \dots = \alpha(f+a_r)$ for random $f\in M_n$, as $q\to \infty$. 
%, or equivalently the probability with the divisor function $\divid(f)$ replaced by $\omega(f)$, the number of distinct prime divisors of $f$.  
That is, setting
$$S_\alpha(\vec a;n,q):=\#\Big\{f\in M_n: \alpha(f+a_1) = \dots = \alpha(f+a_r) \Big\}$$
then 
$$ \Prob\Big\{ f\in M_n: \alpha(f+a_1) = \dots = \alpha(f+a_r)
 \Big\} = \frac 1{q^n} S_\alpha(\vec a;n,q)
\;.
$$

We  treat the case when the arithmetic function $\alpha$  is such that for squarefree $f$, 
the value $\alpha(f)$ depends univalently   on the number $\omega(f)$ of distinct prime (monic irreducible) divisors of $f$, that is for squarefree $f,g\in M_n$, $\alpha(f)=\alpha(g)$ if and only if $\omega(f)=\omega(g)$.  
Examples are:     $\Omega(f)$ the number of all prime divisors, $d(f)$ the number of monic divisors of $f$, and more generally $d_k(f)=\#\{(g_1,\dots, g_k): f=c g_1\dots \cdot g_k, \; g_j \;{\rm monic}, c\in \fq^\times\}$, the number of ways of factoring $f$ as a product of $k$ factors.

The result is  
  \begin{theorem}\label{thm: mult coinc}
Let $\alpha= \omega$, $\Omega$,  or $d_k$. For any  $r$ distinct polynomials $a_1,\dots, a_r\in \fq[T]$ of degree $<n$
 $$
\lim_{q\to\infty} \frac 1{q^n} S_\alpha(\vec a;n,q) \sim \frac {c_r }{( \log n)^{(r-1)/2}}, \quad n\to \infty,  
 $$
with 
$$c_r = \frac 1{(2\pi)^{(r-1)/2}\sqrt{r}} 
\;.$$ 
\end{theorem}
 If we make the translation $X\leftrightarrow q^n$, $\log x \leftrightarrow n$ then we see that we have an analogue for the conjecture \eqref{conj EPS} of  Erd\H{o}s, Pomerance and Sark\H{o}zy (including the same constant).

We prove Theorem~\ref{thm: mult coinc} in two steps: 
For a permutation $\sigma\in S_n$ on $n$ letters, let $\omega_n$ be the number of (disjoint) cycles of $\sigma$. 
We set   
$$ E_r(n)=\frac 1{(n!)^r}\#\{(\sigma_1,\dots,\sigma_r)\in (S_n)^r : \omega_n(\sigma_1) = \dots =\omega_n(\sigma_r)\}  
$$
which is  the probability that $r$ random permutations on $n$ letters all have the same number of cycles.  
\begin{theorem}\label{EPSFqr}
%Let $\alpha=d$, $\omega$ or $\Omega$.  
For any $n\geq 1$, there is some $c_{r,n}>0$ so that for any choice of distinct $ a_1,\dots ,a_r\in \fq[T]$, with $\max_j\deg a_j<n$, 
$$ \Big| S_\omega(\vec a;n,q)- E_r(n) q^n\Big| \leq c_{r,n} q^{n-1/2}
%:=\sum_{\lambda\vdash n} p(\lambda)^2
\;.$$
\end{theorem}
Our key tool for this is the independence of cycle structure for shifted polynomials \cite{ABSR}, see Theorem~\ref{cycle indep thm}. 

 We then show that:
%\begin{proposition}
   \begin{equation}\label{eq:Asymptotics of E_r(n)}
 E_r(n)\sim  \frac {c_r }{( \log n)^{(r-1)/2}}, \quad n\to \infty,\qquad c_r = \frac 1{(2\pi)^{(r-1)/2}\sqrt{r}} 
 \end{equation}
%\end{proposition}
(the case $r=2$ is due to Wilf \cite{Wilf}) which will prove Theorem~\ref{thm: mult coinc}.

\subsection{Locally repeated values of $\varphi$  in $\fq[T]$}

The Euler totient function for   $\fq[T]$ over the finite field $\fq$   is defined as $\varphi(f)=\#\Big(\fq[T]/(f)\Big)^\times$, the number of invertible residues modulo $f$. 
Fix $r\geq 2$. Given $r$ distinct polynomials $a_1[T], ,\dots, a_r[T]\in \fq[T]$, all of degree $\deg a_j< n$, let 
$$ S_\varphi(\vec a;n,q) = \#\{f\in M_n: \varphi(f+a_1) = \varphi(f+a_2)=\dots =\varphi(f+a_r)\}
\;.$$
% (where $M_n\subset \fq[T]$ is the set of monic polynomials of degree $n$). 
%Due to translation invariance, we may take $a_1=0$ if we wish. 

Given a permutation $\sigma\in S_n$, one says that the {\em cycle structure} of $\sigma$ is $(\lambda_1,\dots, \lambda_n)$ if $\sigma$ has $\lambda_i$ cycles of length $i$ 
(the notation $\lambda=(1^{\lambda_1}2^{\lambda_2}\dots n^{\lambda_n})$ is also used in the literature). 
Thus $n=\sum_i i\lambda_i$, and the number of cycles is $\omega(\sigma) = \sum_i \lambda_i$.  
Let $W_r(n)$ be the probability that $r$ random permutations on $n$ letters have the same cycle structure and 
 $$ A_r = \sum_{m=1}^\infty W_r(m) 
\;.
$$
 We have 
$$ A_2 = 4.2634, \quad A_3 = 2.59071\dots, \quad A_4 = 2.23647\dots $$

\begin{theorem}\label{eulerphi thm multiple}
For any choice of  distinct $a_1(T),\dots ,a_r(T)\in \fq[T]$, with $\deg a_j<n$,  
$$\lim_{q\to \infty} \frac 1{q^n}S_\varphi(\vec a;n,q) \sim \frac{A_r}{n^r},\quad {\rm as}\quad n\to \infty 
\;.$$
\end{theorem}
This is in analogy with the conjecture \eqref{conj k even}, once we use  the dictionary $x\leftrightarrow q^n=X$, and $\log x\leftrightarrow n=\log_q X$. 
 
The same result holds if we replace $\varphi$ with any   arithmetic function $\alpha$  for which there is $q_n$ so that if $f,g\in M_n\subset \fq[T]$ are both squarefree, then for all $q>q_n$, $\alpha(f)=\alpha(g)$ is equivalent to $f$ and $g$ having the same cycle structure: $\lambda(f)=\lambda(g)$.
Examples are the sum-of-divisors function $\sigma(f)=\sum_{d\mid f}|d|$ (the sum over monic divisors), where $|d| = \#\fq[T]/(d) = q^{\deg d}$, or more generally $\sigma_s(f) = \sum_{d\mid f}|d|^s$.

We prove Theorem~\ref{eulerphi thm multiple} in two steps. First, we fix $n$, and show: 
\begin{theorem}\label{thm GHP multiple}
For any $n\geq 1$, there is some $C_{r,n}>0$ so that for any choice of distinct $ a_1,\dots ,a_r\in \fq[T]$, with $\max_j\deg a_j<n$, 
$$ \Big| S_\varphi(\vec a;n,q)- W_r(n) q^n\Big| \leq C_{r,n} q^{n-1/2}
\;.$$
\end{theorem}

It remains to asymptotically evaluate $W_r(n)$. 
This was done by Flajolet et al \cite[\S 4.2]{FFGPP} for $r=2$ (see also \cite{BBW}). In \S~\ref{sec:Flajolet} we sketch an adaptation of their method for general $r$, and show:  
\begin{equation}\label{FFGPP gen}
W_r(n) \sim \frac{A_r}{n^r},\quad {\rm as}\; n\to \infty
\;.
\end{equation}
Thus we obtain Theorem~\ref{eulerphi thm multiple}. Appendix~\ref{sec:ron appendix}, by Ron Peled gives a completely different, self-contained, proof of \eqref {FFGPP gen}.  

\subsection{Acknowledgements}
We thank Julio Andrade and Ofir Gorodetsky for discussions on the subject of the paper.

The research of Z.R. was supported by the European Research Council under the European Union's Seventh
Framework Programme (FP7/2007-2013) / ERC grant agreement
n$^{\text{o}}$ 320755, and from the Israel Science Foundation (grant
No. 925/14). Research of R.P. supported by ISF grant 861/15 and by ERC starting grant 678520 (LocalOrder).

\section{Erd\H{o}s and Mirsky over $\fq[T]$}

\subsection{Background on polynomial arithmetic}

For a polynomial $f\in \fq[T]$ of positive degree $n=\deg f$, the cycle structure is $\lambda(f)=(\lambda_1,\dots,\lambda_n)$ if in the decomposition of $f$ into primes (monic irreducibles) $f=c\prod_j P_j$, $c\in \fq^\times$, there are exactly $\lambda_i$ primes of degree $i$. A simple extension of the Prime Polynomial Theorem states that given a partition $\lambda\vdash n$ (so that $\sum_i i\lambda_i=n$), the number of monic polynomials $f\in M_n$ with cycle structure equal to $\lambda$ is 
\begin{equation}\label{PPTcycle}
\#\{f\in M_n:\lambda(f)=\lambda\} = p(\lambda)q^n +O_n(q^{n-1})
\end{equation}
where $p(\lambda)$ is the probability that a random permutation on $n$ letters has cycle structure $\lambda$, which by Cauchy's formula is given by:
$$p(\lambda) = \prod_{j=1}^n \frac 1{j^{\lambda_j} \cdot \lambda_j!}$$

The number of squarefree $f\in M_n$ is $q^n(1-\frac 1q)$ for $n\geq 2$, and hence if $a_1,\dots, a_r\in \fq[T]$ all have degree less than $n$, then as  $q\to \infty$, for all but $O(q^{n-1})$ polynomials $f\in M_n$, all of  $f+a_1,\dots,f+a_r$ are squarefree. 

For $f\in \fq[T]$ of positive degree, let $\omega(f)$ be the number of distinct prime divisors of $f$. 
Let $\alpha$ be an   arithmetic function such that  $\alpha(f)$ depends only on $\omega(f)$ if $f$ is squarefree, and that the dependence is 1-to-1, that is  for squarefree $f$ and $g$, we have 
$\alpha(f)=\alpha(g)$ if and only if $\omega(f)=\omega(g)$. Examples are $\Omega(f)$, the number of all prime divisors, $d(f)$, the number of all (monic) divisors, and more generally $d_k(f)$, the number of all ways of writing $f$ (assumed monic) as a product of $k$ monic polynomials (so  $d=d_2$): 
$$d_k(f) = \#\{(a_1,\dots, a_k)\quad {\rm monic}: f= a_1\cdot \ldots \cdot a_k\}$$
which for squarefree $f$ is given by $d_k(f)=k^{\omega(f)}$. 

For such $\alpha$,  if all of $f+a_1, \dots f+a_r$ are squarefree (which happens for all but $O_{r}(q^{n-1})$ of the $f\in M_n$), then  
$$\alpha(f+a_1)=\dots = \alpha(f+a_r)\quad  \Leftrightarrow \quad \omega(f+a_1)=\dots = \omega(f+a_r)
$$
%So the example of the M\"obius function, which for $f$ squarefree satisfies $\mu(f) = (-1)^{\omega(f)}$ gives a different answer, 
%namely that the probability that $\mu(f+a_1)=\dots = \mu(f+a_r)$ is $1/2^r$. 
%Now for squarefree polynomials, $\Omega(f)=\omega(f)$ and $\divid(f) = 2^{\omega(f)}$. \marginpar{and $d_k(f)=k^{\omega(f)}$}
Thus for such $\alpha$, 
\begin{equation} 
S_\alpha(\vec a ,n,q)  =S_\omega(\vec a, n,q) + O(q^{n-1}) 
\end{equation}
and so in the sequel we may take $\alpha=\omega$.

Our fundamental tool going beyond \eqref{PPTcycle} is the independence of cycle structure 
for shifted polynomials \cite[Theorem 1.4]{ABSR}: 
\begin{theorem}\label{cycle indep thm}
For fixed positive integers $n$, $r$  and    partitions $\lambda^{(1)}\vdash n, \dots$, $\lambda^{(r)}\vdash n$, 
\begin{multline*}
\frac{1}{q^n}\#\{f\in M_n:  \lambda(f+a_1) =\lambda^{(1)}, \cdots, \lambda(f+a_s)=\lambda^{(r)} \}  
= p(\lambda^{(1)}) \cdots p(\lambda^{(r)})
\\
+O_{n,r}\left(q^{-\frac{1}{2}}\right),
\end{multline*}
uniformly for all  distinct polynomials $a_1, \ldots, a_r\in \fq[t]$ of degrees
$\deg(a_i)<n$, as $q\to \infty$.
\end{theorem}
%This is a very strong form of twin-prime type conjectures.  

\subsection{Proof of Theorem~\ref{EPSFqr}} 
 For a permutation $\sigma\in S_n$ on $n$ letters, let $\omega_n$ be the number of cycles of $\sigma$, and  $G_k(n)$ be the probability that a permutation on $n$ letters has $k$ cycles:
$$G_k(n) = \Prob(\omega_n(\sigma)=k )=\frac 1{n!}\#\{\sigma\in S_n: \omega_n(\sigma) = k\}
\;.
$$
Then
$$ E_r(n)  = \sum_{k=1}^n G_k(n)^r 
\;.
$$

Note that $\omega(f)$ may be written in terms of the cycle structure  
$\lambda(f)=(\lambda_1,\dots, \lambda_n)$ of $f$ as $\omega(f) = \omega_n(\lambda(f))=\sum_{j=1}^n \lambda_j$, the number of parts of $\lambda(f)$ (we had earlier used $\omega_n(\sigma)$ for the number of cycles in a permutation $\sigma$). 
Thus 
\begin{equation*}
\begin{split}
\Prob\Big\{&
 f\in M_n: \omega(f+a_1) = \dots = \omega(f+a_r)
 \Big\}  
\\ &= \sum_{k=1}^n \Prob\Big\{f\in M_n: \omega(f+a_1) = \dots = \omega(f+a_r)=k \Big\}
\\ 
& =\sum_{k=1}^n \sum_{\substack{\lambda^{(1)},\dots,\lambda^{(r)}\vdash n\\ \omega_n(\lambda^{(i)} )= k}}
\Prob \Big\{ f\in M_n: \lambda(f+a_1) =\lambda^{(1)},  \dots \lambda(f+a_r)=\lambda^{(r)} \Big\} 
\end{split}
\end{equation*}
where the inner sum is over all $r$-tuples of partitions of $n$ with the same number of parts: $\omega_n(\lambda^{(j)})=k $.

Using independence of cycle structures of $f+a_1, \dots, f+a_r$ (Theorem~\ref{cycle indep thm}), we obtain 
\begin{multline*}
\Prob\Big\{ f\in M_n: \lambda(f+a_1) =\lambda^{(1)},  \dots \lambda(f+a_r)=\lambda^{(r)} \Big\} 
\\
= \prod_{i=1}^r \Prob\{f\in M_n: \lambda(f)=\lambda^{(i)}\} +O_{n,r}(q^{-1/2})
\end{multline*}
and hence
\begin{equation*}
\begin{split}
\Prob\Big\{   
 f\in M_n &:     
\omega(f+a_1) = \dots = \omega(f+a_r)
 \Big\}   
\\ &= \sum_{k=1}^n \sum_{\substack{\lambda^{(1)},\dots,\lambda^{(r)}\vdash n\\ \omega_n(\lambda^{(i)} )= k}}
\prod_{i=1}^r \Prob\{f\in M_n: \lambda(f)=\lambda^{(i)}\} +O_{n,r}(q^{-1/2})
\\
&= \sum_{k=1}^n \Big(\sum_{\substack{\lambda\vdash n\\ \omega_n(\lambda)=k}}\Prob\{f\in M_n: \lambda(f)=\lambda\}
\Big)^r+O_{n,r}(q^{-1/2})
\\ 
&=
\sum_{k=1}^n \Big(\Prob\{f\in M_n: \omega(f)=k  \} \Big)^r+O_{n,r}(q^{-1/2})
\end{split}
\end{equation*}
as $q\to \infty$.

We know that the cycle structure of polynomials of degree $n$ in $\fq[T]$ is modeled by that of random permutations on $n$ letters \eqref{PPTcycle}: 
$$
  \Prob\{f\in M_n: \lambda(f)=\lambda\} = \Prob(\omega_n(\sigma)=k)  +O_n(q^{-1})
$$
and plugging that in will give 
\begin{multline*}
\Prob\Big\{ f\in M_n: \omega(f+a_1) = \dots = \omega(f+a_r)
 \Big\} 
\\
    =   \sum_{k=1}^n G_k(n)^r+O_{n,r}(q^{-1/2})=E_r(n) +O_{n,r}(q^{-1/2})
\end{multline*}
which proves Theorem~\ref{EPSFqr}. \qed

 \section{Coincidences of shifted values of $\varphi$ over $\fq[T]$}

 \subsection{Proof of Theorem~\ref{thm GHP multiple}}
We notice that if $f$ is squarefree, then $\varphi(f)$ only depends on $q$ and on the cycle type of $f$: If $f=\prod P_i$ is a product of distinct primes, with cycle type $\lambda(f) = (\lambda_1,\dots,\lambda_n)$, meaning that it is divisible by exactly $\lambda_j$ primes of degree $j$, so that $\sum_{j=1}^n j\lambda_j=n=\deg f$, then since $\varphi(f) = |f|\prod_i(1-\frac 1{|P_i|})$,  
it follows that 
$$\varphi(f) =   q^n\prod_{j=1}^n(1-\frac 1{q^j})^{\lambda_j}
\;. $$
We define a function $\Phi(\lambda;z)$ on partitions $\lambda\vdash n$   by the above recipe, namely 
$$\Phi(\lambda;z)= \prod_{j=1}^n(1-z^j )^{\lambda_j}$$
so that 
$$\varphi(f) = |f|\Phi\Big(\lambda(f);1/q\Big)
\;.$$
%Note that $\Phi(\lambda;z)$ are polynomials of degree $n$, with constant term $1$. 

Likewise, for the sum-of-divisors function $\sigma(f) = \sum_{d\mid f} |d|$, if $f=\prod P_i$ is a product of distinct primes, with cycle type $\lambda(f) = (\lambda_1,\dots,\lambda_n)$, then 
$$\sigma(f)=\prod_{P\mid f}(|P|+1) = |f| \prod_{i=1}^n (1+\frac 1{q^i})^{\lambda_i} =|f|\Sigma(\lambda(f);1/q) 
% WRONG! \prod_{i=1}^n (\frac{1-\frac 1{q^{i+1}}}{1-\frac 1q})^{\lambda_i}=|f|S(\lambda(f);1/q)
$$
where  for a partition $\lambda\vdash n$, we set 
$$ \Sigma(\lambda;z) := \prod_{i=1}^n (1+z^i)^{\lambda_i}
\;.
$$

Both $\Phi(\lambda;z)$ and $\Sigma(\lambda;z)$ are  polynomials with integer coefficients,  with constant term $1$,  with all  zeros being roots of unity.

 \begin{lemma}\label{lem:distpols}
 If $\lambda\neq \lambda'$ are distinct partitions of $n$, then 
 
 i) the polynomials $\Phi(\lambda;z)$ and $\Phi(\lambda';z)$ are distinct.
 
 ii) There is some $\epsilon_n>0$ so that for all $0<|z|<\epsilon_n$, 
 $$\Phi(\lambda;z)\neq \Phi(\lambda';z)
 \;.$$ 

The same conclusions hold for $\Sigma(\lambda;z)$. 
 \end{lemma}
\begin{proof}
 i) 
If $A(z)=\prod_{j=1}^n(1-z^j)^{a_j}$ and $B(z)=\prod_{j=1}^n(1-z^j)^{b_j}$ with non-negative integers $a_j,b_j$ and $A(z)=B(z)$ as polynomials, we want to show that $a_j=b_j$ for all $j$. 
  We compare logarithmic derivatives (we set $a_i=0=b_i$ if $i>n$):
 $$ -z\frac{A'}{A}(z) = \sum_{m\geq 1} z^m\sum_{i\mid m}ia_i$$
 and likewise for $B$. Therefore for all $m\geq 1$:
 \begin{equation}\label{sum formula}
  \sum_{i\mid m} ia_i = \sum_{i\mid m}ib_i\;.
  \end{equation}
 In particular taking $m=1$ gives $a_1=b_1$. Now we assume by induction that $a_i=b_i$ for $i<I$, then \eqref{sum formula} for $m=I$ gives
 $$ Ia_I + \sum_{\substack{i\mid I\\i<I}} = Ib_I + \sum_{\substack{i\mid I\\i<I}} ib_i $$
 and the inductive hypothesis gives $a_I=b_I$.  
The proof for $S(\lambda;z)$ is similar. 
%We proceed by induction on $n$, the case $n=1$ being trivial. Let $i_{\max} = \max(i:\lambda_i\neq 0)$ and likewise $i_{\max}'= \max(i:\lambda_i'\neq 0)$. Let $N = i_{\max}$. Then the primitive $N$-th root of unity $\zeta_N = e^{2\pi i/N}$ is a zero of $\Phi(\lambda;z)$ with multiplicity $\lambda_{i\max}$, 
%and no primitive root of unity $\zeta_r$ with $m>N$ is a zero of $\Phi(\lambda;z)$. Hence if $\Phi(\lambda;z) \equiv \Phi(\lambda;z)$ then necessarily $i_{\max} = i_{\max}'$ and $\lambda_{i_{\max}} = \lambda_{i_{\max}}'$. 
%Now look at the truncated partitions $\hat \lambda=(\lambda_1,\dots, \lambda_{i_{\max}-1},0,\dots,0)$ and $\hat \lambda'=(\lambda'_1,\dots, \lambda'_{i_{\max}-1},0,\dots,0)$, which are both partitions of $n-\lambda_{i_{\max}}$. 
%By construction, we have 
%$$ \Phi(\hat \lambda,z)\equiv \Phi(\hat \lambda';z)$$
%and therefore by induction $\hat \lambda=\hat \lambda'$, and hence $\lambda=\lambda'$.\end{comment}

  ii)  By part (i), the difference polynomial $F_{\lambda,\lambda''}(z):=\Phi(\lambda;z)-\Phi(\lambda';z)$ is not the zero polynomial if $n>1$. It is a polynomial of degree $\leq n-1$, which vanishes at the origin, since the original polynomials have the same constant term (equal to $1$). Its other zeros are bounded away from the origin, hence part (ii).  
 \end{proof}

  Given distinct $a_1, \dots,  a_r\in \fq[T]$ with $\deg a_j<n$, we  define for an $r$-tuple $\vec \lambda =(\lambda^{(1)},\dots \lambda^{(r)})$ of partitions $\lambda^{(j)}\vdash n$, a function 
%\marginpar{also ask $f+a_j$ all squarefree?}
 \begin{equation*}
 R(\vec \lambda;n,q;\vec a):=\#\left\{ 
\begin{aligned}
f\in M_n: \quad  &\lambda(f+a_1)=\lambda^{(1)}, \dots , \lambda(f+a_r)=\lambda^{(r)} 
\\ 
&\quad f+a_1,\dots f+a_r \quad {\rm all \; squarefree}
\end{aligned}
\right\}
 \;. 
\end{equation*}
 Then 
%\marginpar{First reduce to squarefrees?}
 \begin{equation}\label{R in terms of partitions}
 %\begin{split}
 S_\varphi(\vec a;n,q) 
%&= \sum_{\substack{\vec \lambda\\ \Phi(\lambda^{(1)};1/q) = \dots =\Phi(\lambda^{(m)};1/q)}} \#\{f\in M_n:\lambda(f+a_1)=\lambda^{(1)},\dots, \lambda(f+a_r)=\lambda^{(m)} \}\\ & 
= \sum_{\substack{\vec \lambda\\ \Phi(\lambda^{(1)};1/q) = \dots =\Phi(\lambda^{(r)};1/q)}} R(\vec \lambda;n,q;\vec a)
 + O(q^{n-1})
%\end{split}
\end{equation}
 where the sum is finite, as there are a finite number (depending on $n$) of partitions $\lambda\vdash n$. 
 
 By Lemma~\ref{lem:distpols}(ii), there is some $q_n\gg 1$ so that for all $q>q_n$, if $\lambda\neq \lambda'\vdash n$ then $\Phi(\lambda;1/q) \neq  \Phi(\lambda';1/q)$. Hence for $q>q_n$, we have that the only contribution to the outer sum in \eqref{R in terms of partitions} is the diagonal term $\lambda^{(1)}=\dots = \lambda^{(r)}$:
 $$
 S_\varphi(\vec a;n,q)=\sum_{\lambda\vdash n} R((\lambda,\dots,\lambda);n,q;\vec a)+ O(q^{n-1})
 \;.$$
 
Now we use independence of cycle structures (Theorem~\ref{cycle indep thm}) to write 
 \begin{equation*} 
 \begin{split}
  R((\lambda,\dots,\lambda);n,q;\vec a) &=\#\left\{
\begin{aligned}
f\in M_n:& \quad \lambda(f+a_1)=\lambda,\dots,  \lambda(f+a_r)=\lambda
\\ & \quad  f+a_1,\dots ,f+a_r \quad {\rm squarefree}
\end{aligned}
\right \} 
  \\
  &=q^n \Big(\frac{\#\{f\in M_n: \lambda(f)=\lambda\}}{q^n}  \Big)^r
+O(q^{n-1/2})
\end{split}
\end{equation*}
%\marginpar{check remainder} 
 (uniformly in $\vec a$). By \eqref{PPTcycle}, 
 $$ 
\#\{f\in M_n: \lambda(f)=\lambda\} = p(\lambda)q^n + O_n(q^{n-1})
 $$
 where $p(\lambda)$ is the probability that a random permutation on $n$ letters has cycle structure $\lambda$. 
% $$p(\lambda) = \frac 1{n!}\#\{\sigma\in S_n:\lambda(\sigma)=\lambda\} =\Prob_{S_n}(\lambda(\sigma)=\lambda)\;.$$
  Hence we find that uniformly in $\vec a$,
 $$ 
  S_\varphi(\vec a;n,q) =q^n\sum_{\lambda\vdash n} p(\lambda)^r +O(q^{n-1/2})
 \;.$$
 
 Now note that 
 $$\sum_{\lambda\vdash n} p(\lambda)^r=\Prob\Big((\sigma_1,\dots,\sigma_r)\in (S_n)^r : \lambda(\sigma_1)=\dots = \lambda(\sigma_r)\Big)=:W_r(n)
 $$
is the probability that an $r$-tuple of random permutations in $S_n$ have the same cycle structure,  that is 
 $$  S_\varphi(\vec a;n,q)  =W_r(n) q^n +O(q^{n-1/2})
 $$
(uniformly in $\vec a$). This proves Theorem~\ref{thm GHP multiple}. The case of the sum-of-divisors function is identical.\qed

\subsection{Discussion} 
The crux of the argument is that we are given an  arithmetic function $\alpha$  for which, there is $q_n>1$ so that if $f,g\in M_n$ are both squarefree, then $\alpha(f)=\alpha(g)$ is equivalent to $f$ and $g$ have the same cycle structure:
$$
\alpha(f)=\alpha(g)\leftrightarrow \lambda(f) = \lambda(g), \quad \forall f,g, \in M_n\quad {\rm squarefree},\quad \forall q>q_n
\;.
$$

More generally, consider an  arithmetic function $\alpha$ 
such that for squarefree $f,g\in M_n$,  satisfies: For $q>q_n$,  
$\alpha(f)=\alpha(g)$ if and only if $f$ and $g$ have the same cycle structure: $\lambda(f)=\lambda(g)$. 
Examples are $\varphi$, $\sigma$, and more generally $\sigma_s(f) = \sum_{d\mid f} |d|^s$. 

For such $\alpha$,  given $a_1,\dots, a_r\in \fq[T]$ all of degree less than $n$, for all $f\in M_n$ such that $f+a_1\dots, f+a_r$ are all squarefree, we have 
\begin{equation*}
\alpha(f+a_1)=\dots = \alpha(f+a_r) \; \Leftrightarrow\;  \lambda(f+a_1)=\dots = \lambda(f+a_r) ,\quad \forall q>q_n
\end{equation*}
and therefore
$$
S_\alpha(\vec a;n,q) = S_\varphi(\vec a;n,q) + O_{n,r}(q^{n-1})
$$
which makes the argument go through.

\section{Random permutation theory}

In this section we will prove \eqref{eq:Asymptotics of E_r(n)} and \eqref{FFGPP gen}. For both, the case $r=2$ is known and we verify that similar arguments work in general.

 \subsection{Random permutations with the same number of cycles}
%: Asymptotics of $E_r(n)$

\begin{proposition}\label{prop:Asymptotics of E_r(n)}
 \begin{equation*}
 E_r(n)\sim  \frac {c_r }{( \log n)^{(r-1)/2}}, \quad n\to \infty,\qquad c_r = \frac 1{(2\pi)^{(r-1)/2}\sqrt{r}} 
 \;.
\end{equation*}
 \end{proposition}
 The case $r=2$ can be found in the preprint \cite{Wilf}. 

 \begin{proof}
 Let $f_n(t):=\E(e^{it\omega_n}) $   be  the characteristic function of $\omega_n$, 
 which, by definition, has Fourier expansion
$$
f_n(t) = \sum_{k=0}^\infty \Prob(\omega_n=k) e^{ikt}=\sum G_k(n)e^{ikt}
\;.
$$
Now note that 
\begin{equation*}
E_r(n) =  (f_n\conv \dots \conv f_n)(0)
\end{equation*}
with convolution given by 
$$ f\conv g(x) = \frac 1{2\pi}\int_{-\pi}^\pi f(y)g(x-y)dy
\;.
$$
Indeed, the convolution has Fourier coefficients
$$ \widehat{ f\conv g}(k) = \widehat f (k) \widehat g(k)  $$
 so that the Fourier expansion of the $r$-fold convolution $f_n\conv \dots \conv f_n$ is 
 $$
 f_n\conv \dots \conv f_n(x) = \sum_k G_k(n)^r e^{ikx}
$$
 whose value at $x=0$ is $\sum_k 
 n)^r = E_r(n)$.

As $n\to \infty$, $f_n(t)$ is asymptotic  to\footnote{It is known (see \cite[\S 1.1]{ABT})  
that $f_n(t):=\E(e^{it\omega_n}) = \prod_{j=1}^n (1-\frac 1j +\frac{e^{it}}{j}).$}
\begin{equation}\label{asymp f_n}
f_n(t)\sim \frac{1}{\Gamma(e^{it}) } e^{ \log n(e^{it}-1)}=:g_{\log n}(t)
\;.
\end{equation}
% Indeed, recall that 
%$$\Gamma(z) = \frac{e^{-\gamma z}}{z} \prod_{j=1}^\infty (1+\frac zj)^{-1}e^{z/j}\;.$$
%Now setting $z=e^{it}-1$, we have 
%\begin{equation*}
%\begin{split}
%f_n(t)  &= \prod_{j=1}^n (1+\frac zj) = \prod_{j=1}^n (1+\frac zj) e^{-z/j} \times e^{z\sum_{j=1}^n \frac 1j}
%\\ 
%&\sim \frac{e^{-\gamma z}}{\Gamma(1+z)} e^{z(\log n+\gamma + O(\frac 1n))}\sim \frac{e^{  z \log n}}{\Gamma(1+z)} 
%\;.
%\end{split}
%\end{equation*}

Indeed, we have 
$$ f_n(t) =  \sum_{k=0}^\infty \Prob(\omega_n=k) e^{ikt}  = \sum_{\lambda \vdash n} p(\lambda)e^{i t\sum_{j}\lambda_j}
$$
and hence the generating function $F(z,t) =\sum_{n\geq 0}f_n(t)z^n$ is given by 
\begin{equation*}
\begin{split}
F(z,t) &= \sum_\lambda p(\lambda) e^{it \sum_j\lambda_j}z^{\sum_j j\lambda_j} 
 = \sum_\lambda \prod_{j} \frac{ z^{j\lambda j} e^{it \lambda_j}}{j^{\lambda_j} \cdot \lambda_j !}
\\& =\prod_{j=1}^\infty \sum_{\lambda_j\geq 0}\frac 1{\lambda_j!} \Big( \frac{ z^{j} e^{it }}{j}\Big)^{\lambda^j}
=\exp\sum_{j=1}^\infty \frac{z^je^{it}}{j}
\end{split}
\end{equation*}
that is 
$$ F(z,t)=(1-z)^{-e^{it}}$$
so that $f_n(t)$ is the $n$-th Taylor coefficient of $F(z,t) = (1-z)^{-e^{it}}$. The $n$-th Taylor coefficient of $(1-z)^{-w}$ is asymptotic to 
$$
[z^n](1-z)^{-w} = \frac{\Gamma(n+w)}{\Gamma(w)\Gamma(n+1)}
\sim \frac{n^{w-1}}{\Gamma(w)}, \quad n\to \infty
$$
which gives \eqref{asymp f_n}.

Therefore
$$
E_r(n)\sim   (g_{\log n}\conv \dots \conv g_{\log n})(0)\;.
$$
So we need an asymptotic evaluation, as $\gl \to \infty$, of the $r$-fold convolution
$
%J_r(\gl):=  
 (g_{\gl}\conv \dots \conv g_{\gl})(0)$. 

\begin{lemma}\label{lem Ilambda(r)}
As $\gl\to \infty$, 
$$
(g_{\gl}\conv \dots \conv g_{\gl})(0) %J_r(\gl)
\sim \frac{c_r}{ \gl^{(r-1)/2}}, \qquad  \quad c_r=\frac 1{(2\pi)^{(r-1)/2}\sqrt{r}} 
\;.
$$
\end{lemma}

This will give our claim $ E_r(n)\sim   \frac {c_r }{( \log n)^{(r-1)/2}}$. 
\end{proof}

 \subsection{Proof of Lemma~\ref{lem Ilambda(r)}}
We have
$$(g_{\gl}\conv \dots \conv g_{\gl})(0) = \frac 1{(2\pi)^{r-1}}\int_{[-\pi,\pi]^{r-1}} \prod_{j=1}^r g_{\gl}(t_j) dt_1\dots dt_{r-1} 
$$
where we set 
$$t_r=-(t_1+\dots +t_{r-1})
\;.$$
%\marginpar{ EXPLAIN} 
 
 Outside of the cube $\{( \max_j |t_j|)< \gl^{-0.4}\}$, we have 
 $$\gl \max_j (1-\cos t_j) \gg \gl^{0.2} $$
 and so
 $$
\Big|\prod_{j=1}^r g_{\gl}(t_j) \Big| = \prod_{j=1}^r %|\frac{e^{it_j}(1+e^{it_j})}{\Gamma(2+e^{it_j})} |
\frac {e^{-\gl(1-\cos t_j)}}{|\Gamma(e^{it_j})|}  \ll e^{-c \gl^{0.2}}
 $$
is very rapidly decreasing. So we have
$$
(g_{\gl}\conv \dots \conv g_{\gl})(0) \sim 
 \frac 1{(2\pi)^{r-1}}
 \int_{ ( \max_j |t_j|)< \gl^{-0.4}}  \prod_{j=1}^r g_{\gl}(t_j) dt_1 \dots dt_{r-1} 
\;.
$$

For $|t|<\gl^{-0.4}$ we may write 
$$
g_{\gl}(t)=\frac{1}{\Gamma(e^{it}) } e^{\gl(e^{it}-1)}=
e^{i\gl t-\gl t^2/2}\Big( 1+O(\gl^{-0.2}) \Big)
\;.
$$
Hence in this small cube, 
\begin{multline*}
\int_{ ( \max_j |t_j|)< \gl^{-0.4}}  \prod_{j=1}^r g_{\gl}(t_j) dt_1 \dots dt_{r-1} 
\\
\sim  \int_{ ( \max_j |t_j|)< \gl^{-0.4}}
   e^{i\gl\sum_j t_j} e^{-\frac \gl 2\sum_j t_j^2} 
 dt_1 \dots dt_{r-1} 
 \\
 = \int_{ ( \max_j |t_j|)< \gl^{-0.4}}
     e^{-\frac \gl 2\sum_j t_j^2} 
 dt_1 \dots dt_{r-1} 
\end{multline*}
since $\sum_{j=1}^r t_j=0$. 

Changing variables $u=\sqrt{\frac \gl 2}\;t$  gives
\begin{equation*}
\int_{ ( \max_j |t_j|)< \gl^{-0.4}}
  e^{-\frac \gl 2\sum_j t_j^2} 
dt_1 \dots dt_{r-1} 
 \sim \Big(\sqrt{\frac 2\gl}\Big)^{r-1} 
 \int_{\R^{r-1}}   \prod_{j=1}^r e^{-u_j^2}  \cdot du_1\dots du_{r-1}
\end{equation*}
where $u_r = -(u_1+\dots +u_{r-1})$. 
 Therefore we find 
$$
 (g_{\gl}\conv \dots \conv g_{\gl})(0) \sim \frac{c_r}{\gl^{(r-1)/2}}
 $$
with 
$$ c_r = \frac{(\sqrt{2})^{r-1}}{(2\pi)^{r-1}}  \int_{\R^{r-1}}   \prod_{j=1}^r e^{-u_j^2}du_j
\;.
$$
It remains to determine the  Gaussian integral. This is precisely the $r$-fold convolution of $e^{-u^2}$ with itself (convolution over $\R$): 
$$
 \int_{\R^{r-1}}   \prod_{j=1}^r e^{-u_j^2}du_j
  = (e^{-u^2}\conv_\R \dots \conv_\R e^{-u^2})(0)
 \;.
 $$
Using Fourier inversion, this equals 
$$
(e^{-u^2}\conv_\R \dots \conv_\R e^{-u^2})(0) = \int_{-\infty}^\infty
\widehat{(e^{-u^2}\conv_\R \dots \conv_\R e^{-u^2})}  (x) dx  = 
\int_{-\infty}^\infty \Big(\widehat{e^{-u^2}} (x)\Big)^rdx
$$
Now the Fourier transform of $e^{-u^2}$ is 
$$
\widehat{e^{-u^2}} (x) = \int_{-\infty}^\infty e^{-u^2}e^{-2\pi i ux}du = \sqrt{\pi}
e^{-\pi^2x^2}
\;.
$$
Hence 
$$
(e^{-u^2}\conv_\R \dots \conv_\R e^{-u^2})(0)= 
\int_{-\infty}^\infty \Big( \sqrt{\pi}
e^{-\pi^2x^2}\Big)^rdx =\pi^{r/2} \int_{-\infty}^\infty e^{-r\pi^2x^2}dx =\frac{\pi^{(r-1)/2}}{\sqrt{r}}
\;.
$$
Thus 
$$c_r=  \frac{(\sqrt{2})^{r-1}}{(2\pi)^{r-1}}\frac{\pi^{(r-1)/2}}{\sqrt{r}}
=\frac 1{(2\pi)^{(r-1)/2}\sqrt{r}} 
$$
concluding the proof. \qed

\subsection{Random permutations with the same cycle structure: Proof of \eqref{FFGPP gen}}\label{sec:Flajolet}
 Let $W_r(n)$ be the probability that $r$ random permutations on $n$ letters have the same cycle structure. 
Arguing as in \cite[\S 4.2]{FFGPP} (who treat the case $r=2$, for which a completely different argument is also given in \cite{BBW}) shows that  
\begin{equation}\label{FFGPP gen*}
W_r(n) \sim \frac{A_r}{n^r},\quad {\rm as}\; n\to \infty
\end{equation}
where 
$$ A_r = \sum_{n=1}^\infty W_r(n)= \prod_{k=1}^\infty I_r(\frac 1{k^r})
$$
%The proof of \eqref{FFGPP gen} gives 
%$$ A_m = \prod_{k=1}^\infty I_k(\frac 1{k^m})$$
with  
$$ I_r(y) = \sum_{j=0}^\infty \frac{y^j}{(j!)^r}= \, _0F_{r-1}(\; ;\underbrace{1,\dots,1}_{r-1};y) $$
is the generalized hypergeometric function.  
Thus we evaluate 
$$ A_2 = 4.2634, \quad A_3 = 2.59071\dots, \quad A_4 = 2.23647\dots $$

Indeed,  the generating function of the probability that $r$ random permutations share the same cycle structure is 
$$
W^{(r)}(z) :=\sum_{n\geq 0} W_r(n)  z^n  =\sum_{n\geq 0} \sum_{\lambda \vdash n} \frac{z^{\lambda_1+2\lambda_2+\dots}}{\prod_i i^{r\lambda_i} (\lambda_i!)^r} 
=\prod_{k\geq 1} \Big( \sum_{\lambda_k\geq 0} \frac{z^{k \lambda_k}}{k^{r \lambda_k}(\lambda_k!)^r} \Big) 
\;.
$$  
Thus 
$$
W^{(r)}(z) = \prod_{k\geq 1} I_r(\frac{z^k}{k^r})
$$
with 
$$ I_r(y) = \sum_{j=0}^\infty \frac{y^j}{(j!)^r}= \, _0F_{r-1}(\; ;\underbrace{1,\dots,1}_{r-1};y) 
\;.
$$
We have 
$$H_r(y):=\log I_r(y) = \sum_{\ell \geq 1} h^{(r)}_\ell  y^\ell  = y+O(y^2)
$$
and so 
$$
W^{(r)}(z) = \exp\Big( \sum_{k\geq 1} H_r(\frac{z^k}{k^r} )\Big) = \exp\Big( \sum_{\ell\geq 1} h^{(r)}_\ell {\rm Li}_{r \ell}(z^\ell) \Big)
$$
where 
$${\rm Li}_\nu(z) = \sum_{j\geq 1} \frac{z^j}{j^\nu}
$$
is the polylogarithm function. 

The expansion converges in the unit disk $|z|<1$, and the dominant singularity on the unit circle $|z|=1$ is at $z=1$, where 
$${\rm Li}_r(z) \sim  \frac{(-1)^{r-1}}{(r-1)!} (1-z)^{r-1} \log \frac 1{1-z},\quad z\to 1^-
$$
so that
$$
W^{(r)}(z) \sim W^{(r)}(1)\Big( 1+  \frac{(-1)^{r-1}}{(r-1)!} (1-z)^{r-1} \log \frac 1{1-z}+\dots \Big), \quad z\to 1^-
\;.
$$
By the argument of \cite{FFGPP}, the $n$-th coefficient  in the expansion of $W^{(r)}(z)/W^{(r)}(1)$ at $z=0$ is asymptotic to that of $\frac{(-1)^{r-1}}{(r-1)!} (1-z)^{r-1} \log \frac 1{1-z}$, which is   
$$
\frac{(-1)^{r-1}}{(r-1)!}  [z^n](1-z)^{r-1} \log \frac 1{1-z}\sim    \frac 1{n^r},\quad n\to \infty
$$
and hence  
$$W_n^{(r)}\sim \frac {W^{(r)}(1)}{n^r},\quad n\to \infty$$
which is the claimed result. \qed

%%%%%%%%%%%%% BEGIN APPENDIX 

\appendix
\section{Permutations with the same cycle structure\\ by Ron Peled} \label{sec:ron appendix}

For integer $n\ge 1$, a vector $\lambda = (\lambda_1,\ldots,\lambda_n)$ of non-negative integers is said to be a \emph{partition of $n$}, denoted $\lambda \vdash n$, if $\sum_{j=1}^n j\lambda_j=n$.
The cycle structure of a permutation on $n$ letters is the partition of $n$ given by $\lambda = (\lambda_1,\dots, \lambda_n)$, where $\lambda_j$ is the number of cycles of length $j$.

Let $W_r(n)$ be the probability that $r$ uniformly random and independent permutations on $n$ letters have the same cycle structure. Define also $W_r(0):=1$.
In this section we prove that
\begin{theorem}\label{thm:equal cycle}
For $r\geq 2$ integer,
$$ W_r(n)\sim \frac{A_r}{n^r}, \quad n\to \infty$$
where
$$ A_r := \sum_{m=0}^\infty W_r(m) .
$$
\end{theorem}
%\marginpar{Does the sum start at $m=0$?}

Fix an integer $r\ge 2$. Cauchy's formula says the probability that a uniformly sampled permutation on $n$ letters has cycle structure $\lambda$ is
\begin{equation}\label{eq:Cauchy_formula}
p(\lambda) := \prod_{j=1}^n\frac 1{j^{\lambda_j}\; \lambda_j!}
\end{equation}
so that
\begin{equation}\label{eq:sum_of_probabilities}
\sum_{\lambda\vdash n} p(\lambda) = 1
\end{equation}
and
$$
W_r(n) = \sum_{\lambda\vdash n} p(\lambda)^r .
$$
We denote by $\omega_n(\lambda)$ the number of cycles in $\lambda$:
$$\omega_n(\lambda) := \sum_j \lambda_j ,
$$
%(we had earlier used $\omega_n(\sigma)$ for the number of cycles in a permutation $\sigma$) 
and by $T(\lambda)$ the length of the longest cycle:
$$T(\lambda) := \max\left( j: \lambda_j\neq 0\right) .
$$

For $n\ge 1$ we set
$$ L=L(n):=\log_2(2n^3), \qquad a=a(n) :=  \frac{1}{3}n^{2/3}
$$
($\log_2$ is the base $2$ logarithm) and from the set of cycle structures $\lambda\vdash n$ we form three subsets:
\begin{itemize}

\item $ \I = \I(n)$\;:  $\omega_n(\lambda)> L$;
\\
\item $\II = \II(n)$\;\,:  $\omega_n(\lambda)\leq L $ and $n/L\leq T(\lambda)<n-a$ ;
\\
\item $\III = \III(n)$\;\,\,: $n-a\leq T(\lambda)\leq n$.

\end{itemize}
Further, set
$$\sI = \sI(n) := \sum_{\lambda \in \I} p(\lambda)^r$$
and likewise for $\sII$, $\sIII$.

  We claim that $\I\cup \II \cup \III$ exhausts all cycle structures $\lambda\vdash n$.
Indeed, note that $\lambda\vdash n$ means $\sum_j j\lambda_j=n$, and the sum takes place only over $j\leq T(\lambda)$ by definition. Hence
$$n = \sum_{j=1}^{T(\lambda)} j\lambda_j \leq T(\lambda)\sum_j \lambda_j = T(\lambda)\omega_n(\lambda)$$
so that $ T(\lambda)\geq n/\omega_n(\lambda)$.
Now if $\lambda\notin \I$, i.e.\  $\omega_n(\lambda)\leq L$, then
$$ T(\lambda)\geq n/\omega_n(\lambda) \geq n/L$$
so that $\lambda \in \II\cup \III$. Note that we omitted  the requirement that $\omega_n(\lambda)\leq L$ from $\III$, so we do not get a disjoint union: $\III\cap \I\neq \emptyset$. Nonetheless, we have
\begin{equation}\label{bds for W_r(n)}
\sIII\leq W_r(n) \leq  \sI + \sII +\sIII
\end{equation}
and we will see that the dominant contribution to $W_r(n)$ for large $n$ will be from $\sIII$.

\begin{lemma}\label{lem:sI} For $n\ge 1$,
$$ \sI \leq \frac{1}{n^{3(r-1)}}.
$$
\end{lemma}
\begin{proof}
We use, for $\lambda_j\geq 0$, that
$$ \frac 1{j^{\lambda_j} \; \lambda_j!} \leq \frac {2}{2^{\lambda_j}}$$
(where the factor $2$ is only needed when $j=1$ and $1\le\lambda_1\le 3$). Therefore
$$
p(\lambda) = \prod_{j=1}^n \frac 1{j^{\lambda_j} \; \lambda_j!}
\leq
2 \prod_{j=1}^n \frac 1{2^{\lambda_j}} = \frac{2}{2^{\omega_n(\lambda)}} ,
$$
from which we deduce, using \eqref{eq:sum_of_probabilities} and the definition of $L$, that
\begin{equation*}
\begin{split}
\sI &= \sum_{\substack{\lambda \vdash n\\ \omega_n(\lambda)> L}} p(\lambda)^r
\\
&\leq
  \sum_{\substack{\lambda \vdash n\\ \omega_n(\lambda) > L}} p(\lambda) \left(\frac 2{2^{\omega_n(\lambda) }}\right)^{r-1}\\
&\leq
\left(\frac{2 }{2^{ L}}\right)^{r-1} \sum_{\lambda \vdash n} p(\lambda) =
\left(\frac{2 }{2^{ L}}\right)^{r-1} = \frac{1}{n^{3(r-1)}}.\qedhere
\end{split}
\end{equation*}
\end{proof}

\begin{lemma}\label{lem:main term for sIII} For $n\ge 1$,
\begin{equation*}
\sIII  = \sum_{0\leq m\leq a } \frac 1{(n-m)^r} W_r(m) .
\end{equation*}
\end{lemma}
\begin{proof}
Note that for  $\lambda \in \III$,  if $t=T(\lambda)\geq n-a$, then since $a<n/2$ there is a {\em unique} cycle of length $t$. Thus
$$\lambda = (\hat \lambda,\overbrace{1}^{ {\rm position}\; t},0,\dots)$$
with $\hat \lambda$ a partition of $n-t$, and then $p(\lambda) =  p(\hat \lambda)\cdot \frac 1t$ by Cauchy's formula~\eqref{eq:Cauchy_formula} (where for $t=n$, $\hat \lambda$ is empty and we define $p(\hat\lambda):=1$).  Hence
\begin{equation*}
\begin{split}
\sIII &= \sum_{n-a\leq t\leq n }\sum_{\substack{ \lambda \in \III\\ T(\lambda)=t}} p(\lambda)^r
\\
&=\sum_{n-a\leq t\leq n } \frac 1{t^r} \sum_{\hat \lambda \vdash n-t} p(\hat \lambda)^r
\\& = \sum_{n-a\leq t\leq n } \frac 1{t^r} W_r(n-t)
\end{split}
\end{equation*}
(recalling that $W_r(0)=1$). Changing variables to $m=n-t$ gives our statement.
\end{proof}

We next want to use induction to give upper bounds for $W_r(n)$ and an asymptotic bound for $\sII$.

\begin{lemma}
There is a constant $C_r>0$ so that for all $n\geq 1$,
\begin{equation}\label{upper bound for W_r(n)}
W_r(n) \leq  \frac {C_r}{n^r}.
\end{equation}
In addition
\begin{equation}\label{asymptotic bound for II}
\sII = o\left(\frac{1}{n^r}\right)\quad\text{as }n\to\infty.
\end{equation}
\end{lemma}

%\begin{lemma}
%There is a constant $C_r>0$ so that for all $n\geq 1$,
%\begin{equation}\label{upper bound for II}
%\sII \leq  \frac {C_r}{n^r} \cdot \frac{(3L)^r}{n^{1/3}} ,   % \frac{(\log n)^{2r}}{n^{r+\tfrac 23 r-1}} ,
%\qquad
%\sIII \leq  \frac {C_r}{n^r}
%\end{equation}
%and
%\begin{equation}\label{upper bound for W_r(n)}
%W_r(n)\leq \frac{3C_r}{n^r} .
%\end{equation}
%\end{lemma}

\begin{proof}
Fix an integer $N_r\ge 1$ for which (recalling that $r\ge 2$)
\begin{equation}\label{eq:N bound}
  \sum_{m=N_r}^\infty \frac{1}{m^r} \le \frac{1}{6}\cdot\left(\frac{2}{3}\right)^r\quad\text{and}\quad \sup_{k\ge N_r}\frac{(3L(k))^r}{k^{\frac{2}{3}r-1}}\le \frac{1}{3}.
\end{equation}
Let $C_r>0$ be a sufficiently large constant, satisfying several lower bounds imposed below. We take $C_r\ge (N_r)^r$ so that \eqref{upper bound for W_r(n)} is satisfied for $1\le n\le N_r$, as $W_r(n)\le 1$ for all $n$.

%whose value will be specified (implicitly) later. Note that the bounds in \eqref{upper bound for II} are satisfied for $n=1$ as $C_r\ge 1$ (since $\sII=0, \sIII=1$ for $n=1$). We proceed to establish that \eqref{upper bound for II} holds for $n\ge 2$ and \eqref{upper bound for W_r(n)} holds for $n\ge 1$ using induction.
%
%%Recall that we set
%%$$ L=L(n)=1+(\log n)^2, \qquad a=a(n) = \lceil n^{2/3} \rceil .$$
%Let
%$$B_r = \sup_{a\geq 1}  a^{r-1} \sum_{m=a}^\infty \frac 1{m^{r}} .
%$$
%Choose $N_r\geq 1$ satisfying
%\begin{align}
%%N_r&\geq 10
%%\\
%N_r &\geq (100  B_r)^{1/(r-1)} ; \label{req 2 for T}
%\\
%\left(\frac{2}{2^{L(n)}}\right)^{r-1} & \leq \frac 1{10} , \qquad \forall n\geq N_r  ;\label{req 3 for T}
%\\
%\frac{B_r L(n)^r}{n^{\frac 23(r-1)}} & \leq \frac 1{10 } , \qquad \forall n\geq N_r  .
%\label{req 4 for T}
%\end{align}
%Now choose $C_r>0$ so that  \eqref{upper bound for II} and \eqref{upper bound for W_r(n)} hold for all $n\leq N_r$, and
%\begin{equation}\label{choose C_r}
%C_r  \geq 100 \sum_{m=0}^{N_r-1} W_r(m) .
%\end{equation}

Let $k>N_r$ and assume by induction that \eqref{upper bound for W_r(n)} holds with $1\le n<k$. We proceed to establish \eqref{upper bound for W_r(n)} with $n=k$. To this end, we first give upper bounds for $\sI(k)$, $\sII(k)$ and $\sIII(k)$.

By Lemma~\ref{lem:sI}, using that $3(r-1)> r$ and taking $C_r\ge 3$,
\begin{equation}\label{eq:upper bound for sI}
\sI(k)\leq  \frac{1}{k^{3(r-1)}} \leq  \frac {1}{3}  \frac {C_r}{k^r}
\end{equation}

For $\lambda \in \II(k)$, such that $T(\lambda) = t$,
 note that since $\lambda_t\geq 1$,
$$ \frac 1{t^{\lambda_t} \; \lambda_t!} \leq \frac 1t \cdot \frac 1{t^{\lambda_t-1}\;(\lambda_t-1)!}
$$
%we can write
%$$\lambda = (\hat \lambda,\overbrace{\lambda_t}^{t},0,\dots)$$
%with $\lambda_t\geq 1$ and $\hat \lambda$ a partition of $n-t\lambda_t$: $\hat \lambda \vdash n-t\lambda_t$,
and then by Cauchy's formula \eqref{eq:Cauchy_formula},
$$p(\lambda) =  \prod_{j=1}^{t-1} \frac 1{j^{\lambda_j}\;\lambda_j!} \cdot \frac 1{t^{\lambda_t}\;\lambda_t!} \leq
\frac 1t \cdot \left( \prod_{j=1}^{t-1} \frac 1{j^{\lambda_j}\;\lambda_j!} \right)\cdot  \frac 1{t^{\lambda_t-1}\;(\lambda_t-1)!}
=\frac 1t p(\tilde \lambda)$$
 where $\tilde \lambda =  (\lambda_1,\dots, \overbrace{\lambda_t-1}^{t},0,\dots)$ is a partition of $k-t$.
Hence
\begin{equation*}
\begin{split}
\sII(k) &=\sum_{\frac{k}{L}\leq t<k-a} \sum_{\substack{\lambda\in \II\\ T(\lambda)=t}} p(\lambda)^r
\\
&\leq  \sum_{\frac{k}{L}\leq t<k-a} \frac 1{t^r} \sum_{\tilde \lambda \vdash k-t} p(\tilde \lambda)^r
\\
& =  \sum_{\frac{k}{L}\leq t<k-a} \frac 1{t^r}  W_r(k-t) .
\end{split}
\end{equation*}
Using the induction hypothesis for $n=k-t$  we obtain
\begin{equation}\label{eq:upper bound for sII}
\begin{split}
\sII(k) &\leq   \sum_{\frac{k}{L(k)}\leq t<k-a(k)} \frac 1{t^r} \frac{C_r}{(k-t)^r}\\
&\le C_r \, k \left(\frac{L(k)}{k\, a(k)}\right)^r = \frac{C_r}{k^r}\frac{(3L(k))^r}{k^{\frac{2}{3}r-1}} \le \frac{1}{3}\frac{C_r}{k^r}
\end{split}
\end{equation}
where we substituted the definition of $a(k)$ and applied \eqref{eq:N bound}.

Next, using Lemma~\ref{lem:main term for sIII},
\begin{equation*}
\sIII(k) \leq \frac 1{(k-a)^r} \sum_{0\leq m\leq a} W_r(m)\leq  \left(\frac{3}{2k}\right)^r  \left( N_r +  \sum_{N_r\leq m\leq a}  \frac{C_r}{m^r} \right)
\end{equation*}
where we used the fact that $W_r(m)\le 1$ for all $m$ and the induction hypothesis for $N_r\leq n=m\leq a$. Applying \eqref{eq:N bound} and taking $C_r\ge 6 \left(\frac{3}{2}\right)^r N_r$, we conclude that
\begin{equation}\label{eq:upper bound for sIII}
\sIII(k)\leq  \left(\frac{3}{2k}\right)^r  \left( N_r +  \frac{1}{6}\left(\frac{2}{3}\right)^r C_r \right)\le \frac{1}{3}\frac {C_r}{k^r} .
\end{equation}

Finally, applying \eqref{bds for W_r(n)} and substituting \eqref{eq:upper bound for sI}, \eqref{eq:upper bound for sII} and \eqref{eq:upper bound for sIII} we conclude that
$$W_r(k)\leq \sI(k)+\sII(k) +\sIII(k) \le \frac{C_r}{k^r}.
$$
finishing the proof by induction of \eqref{upper bound for W_r(n)}.

To see \eqref{asymptotic bound for II}, note that as the bound \eqref{eq:upper bound for sII} is now verified for all large $k$ and as $r\ge 2$, we have
\begin{equation*}
  \sII(k) \leq \frac{C_r}{k^r}\frac{(3L(k))^r}{k^{\frac{2}{3}r-1}} = o\left(\frac{1}{k^r}\right)\quad\text{as $k\to\infty$}.\qedhere
\end{equation*}
\end{proof}

We are now in position to obtain the asymptotics of $\sIII$:
\begin{equation}\label{asymp of sIII}
\sIII  \sim \frac{A_r}{n^r},\qquad A_r = \sum_{m=0}^\infty W_r(m) .
\end{equation}
Indeed, by Lemma~\ref{lem:main term for sIII},
\begin{equation*}
\sIII  = \sum_{0\leq m\leq a } \frac 1{(n-m)^r} W_r(m)\sim \frac 1{n^r} \sum_{0\leq m\leq a }  W_r(m)
\end{equation*}
since $(n-m)^r\sim n^r$ uniformly for $0\le m\leq a=o(n)$. We now extend the sum, using our upper bound \eqref{upper bound for W_r(n)} and the fact that $a(n)\to\infty$, and obtain~\eqref{asymp of sIII}.
 
We can now prove Theorem~\ref{thm:equal cycle}: By  \eqref{bds for W_r(n)},
$$ W_r(n) = \sIII + O\Big( \sI+\sII\Big) \sim \frac{A_r}{n^r}$$
using \eqref{asymp of sIII}, Lemma~\ref{lem:sI} (recalling that $r\ge 2$ so that $3(r-1)>r$) and \eqref{asymptotic bound for II}.

\end{document}